\documentclass[12pt,psamsfonts]{amsart}
\usepackage{amsmath}
\usepackage{amsthm}
\usepackage{amssymb}
\usepackage{amscd}
\usepackage{amsfonts}
\usepackage{amsbsy}
\usepackage{graphicx}
\usepackage[dvips]{psfrag}
\usepackage{array}
\usepackage{color}
\usepackage{epsfig}
\usepackage{url}

\newcolumntype{L}{>{\displaystyle}l}
\newcolumntype{C}{>{\displaystyle}c}
\newcolumntype{R}{>{\displaystyle}r}

\newcommand{\N}{\ensuremath{\mathbb{N}}}

\newcommand{\g}{\gamma}

\newcommand{\f}{\varphi}
\newcommand{\al}{\alpha}

\newcommand{\C}{\ensuremath{\mathcal{C}}}

\newcommand{\sgn}{\mathrm{sign}}

\newtheorem {theorem} {Theorem}%[section]

\newtheorem {corollary} [theorem]{Corollary}
\newtheorem {lemma}  [theorem]{Lemma}

\newtheorem {conjecture} [theorem]{Conjecture}

\begin{document}
\renewcommand{\arraystretch}{1.5}

\title[A simple solution to the Braga--Mello conjecture]
{A simple solution to the\\ Braga--Mello conjecture}

\author[D.D. Novaes and E. Ponce]
{Douglas D. Novaes$^{1,2}$ and Enrique Ponce$^3$}

\address{$^1$ Departament de Matem\`{a}tiques,
Universitat Aut\`{o}\-noma de Barce\-lona, 08193 Bellaterra, Barcelona,
Catalonia, Spain} \email{ddnovaes@mat.uab.cat}

\address{$^{2}$  Departamento de Matem\'{a}tica, Universidade Estadual de Campinas, CP 6065, 13083-859, Campinas, SP, Brazil.} \email{ddnovaes@ime.unicamp.br}

\address{$^{3}$  Departamento de Matem\'{a}tica Aplicada II, Escuela T\'ecnica Superior de Ingenier\'{\i}a,  Camino de los Descubrimientos s.n., 41092 Sevilla, Spain} \email{eponcem@us.es}



\maketitle

\begin{abstract}
Recently Braga and Mello conjectured that for a given $n\in\N$ there is a piecewise linear system with two zones in the plane with exactly $n$ limit cycles. In this paper we prove a result from which the conjecture is an immediate consequence. Several explicit examples are given where location and stability of limit cycles are provided.
\end{abstract}

\section{Introduction and statement of the main results}

The computation of upper bounds for the number of limit cycles in all possible configurations within the family of planar piecewise linear differential systems with two zones has been the subject of some recent papers. Assuming that the separation boundary is a straight line, Han and Zhang \cite{HanZhang} conjectured in 2010  that for such planar piecewise linear systems there can be at most two limit cycles. However,  Huan and Yang \cite{HuanYang} promptly gave a negative answer to this conjecture by means of a numerical example with three limit cycles under a focus-focus configuration. Such counter-intuitive example led researchers to look for rigorous proofs of this fact, see \cite{LlibrePonce2012} for a computer-assisted proof, and \cite{FPT2014} for an analytical proof under a more general setting.

When the boundary between the two linear zones is not a straight line any longer, it is possible to obtain more than three limit cycles. Thus, by resorting to the same example introduced in \cite{HuanYang} and analyzed in \cite{LlibrePonce2012},  Braga and Mello  studied in \cite{BM} some members of the following class of discontinuous piecewise linear differential system  with two zones
\begin{equation}\label{s1}
X'=\left\{
\begin{array}{l}
G^- X\quad\textrm{if}\quad H(X,p)<0,\vspace{0.2cm}\\
G^+ X\quad\textrm{if}\quad H(X,p)\geq 0,
\end{array}
\right.
\end{equation}
where the prime denotes derivative with respect to the independent variable $t$, $p$ is a parameter vector, $X=(x,y)$, and
\[
G^{\pm}=\left(\begin{array}{cc} g_{11}^{\pm}&g_{12}^{\pm}\\ g_{21}^{\pm}&g_{22}^{\pm}\end{array}\right)
\]
are matrices with real entries satisfying the following assumptions:
\begin{itemize}
\item[$(H1)$] $g_{12}^{\pm}<0$,
\item[$(H2)$] $G^-$ has complex eigenvalues with negative real parts while $G^+$ has complex eigenvalues with positive real parts, and
\item[$(H3)$] the function $H$ is at least continuous.
\end{itemize}
After using some broken line as the boundary between linear zones, Braga and Mello in \cite{BM} put in evidence the important role of the separation boundary in the number of limit cycles. They obtained examples with different number of limit cycles, and accordingly stated the following conjecture in \cite{BM}.
\begin{conjecture}\label{c1}
Given $n\in\N$ there is a piecewise linear system with two zones in the plane with exactly $n$ limit cycles.
\end{conjecture}

In this paper, we prove that the above conjecture is true by showing how to perturb a rather simple vector field in order to get as many limit cycles as wanted. We provide several concrete examples. Furthermore, we show that the involved methodology allow us to locate the position of the limit cycles and determine their stability. 

We start from the normal form given in \cite{FPT} for systems of kind \eqref{s1} and after selecting an appropriate value for $\g>0$, we take 
\begin{equation}\label{vf}
G^{\pm}=\left(\begin{array}{cc} {\pm}2\g&-1\\ \g^2+1&0\end{array}\right),
\end{equation}
and
\begin{equation}\label{vfborder}
H(X)=\left\{
\begin{array}{ll}
x &\textrm{if} \quad y\leq 0,\vspace{0.2cm}\\
x-h(y)& \textrm{if}\quad y> 0,
\end{array}
\right.
\end{equation}
where $h(y)$ is a $\C^1$ function such that $h(0)=0$. We also assume that for $y>0$ the following hypotheses:
\begin{itemize}
\item[$(H1')$] $|h(y)|<y/\g$,
\item[$(H2')$] $h(y)(2\g-(1+\g^2)h'(y))<y$, and
\item[$(H3')$] $h(y)(2\g+(1+\g^2)h'(y))>-y$.
\end{itemize}
%\begin{equation}\label{ine1}
%\begin{array}{l}
%|h(y)|<y/\g,\quad h(y)(2\g-(1+\g^2)h'(y))<y,\quad \textrm{and}\vspace{0.2cm}\\
%h(y)(2\g+(1+\g^2)h'(y))>-y.
%\end{array}
%\end{equation}
As shown in Section \ref{sec2}, Hypothesis $H1'$ is just assumed to facilitate the computation of solutions, whilst Hypotheses $H2'$ and $H3'$ allow us to assure that the two linear vector fields can be concatenated across the discontinuity manifold in the natural way, so avoiding the existence of sliding sets, see  \cite{KRG}.

\begin{figure}[t]\label{fig1}
\begin{minipage}{0.45\linewidth}
\centering
\includegraphics[width=7cm]{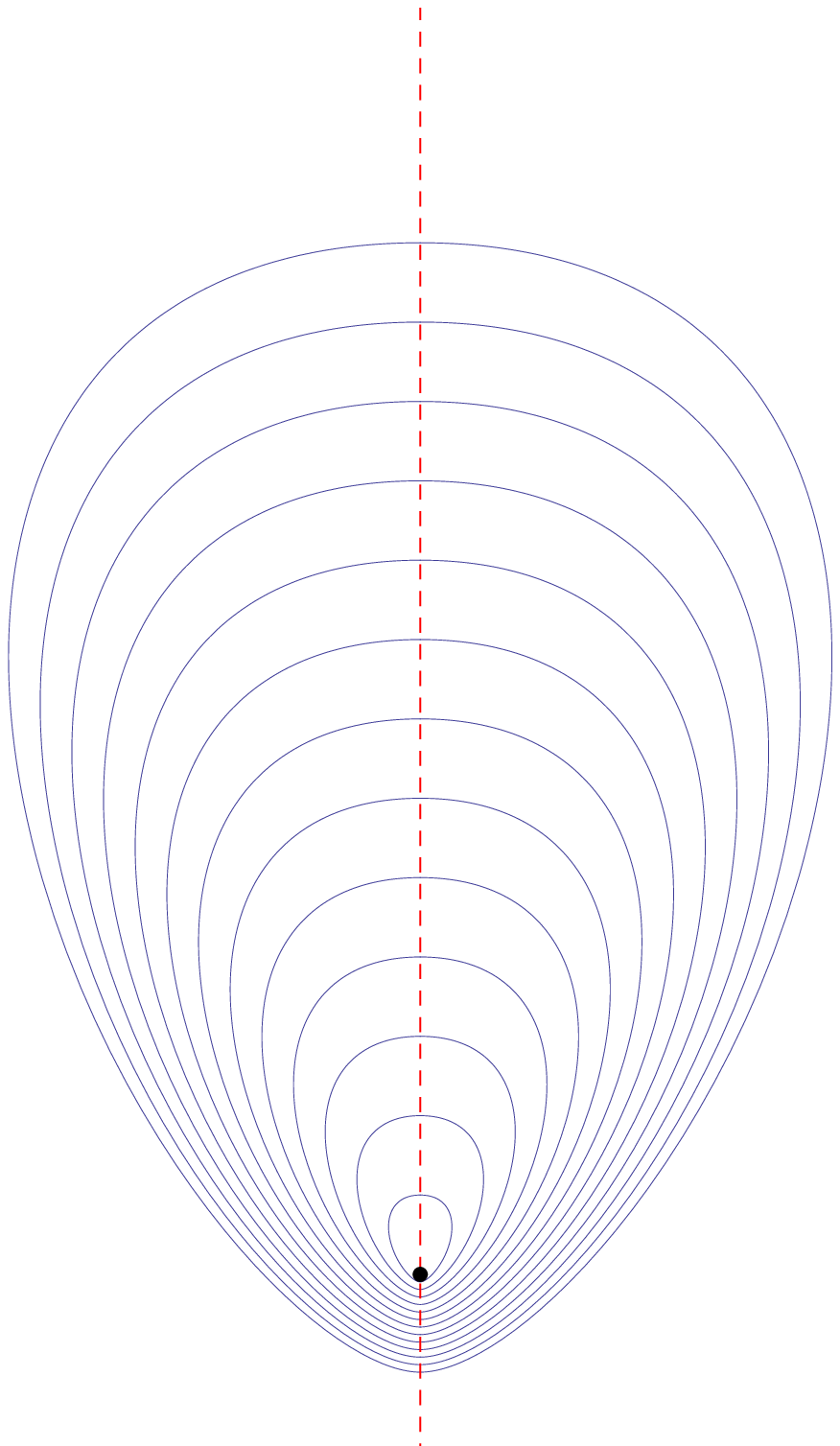}
\end{minipage}
\begin{minipage}{0.45\linewidth}
\centering
\includegraphics[width=7cm]{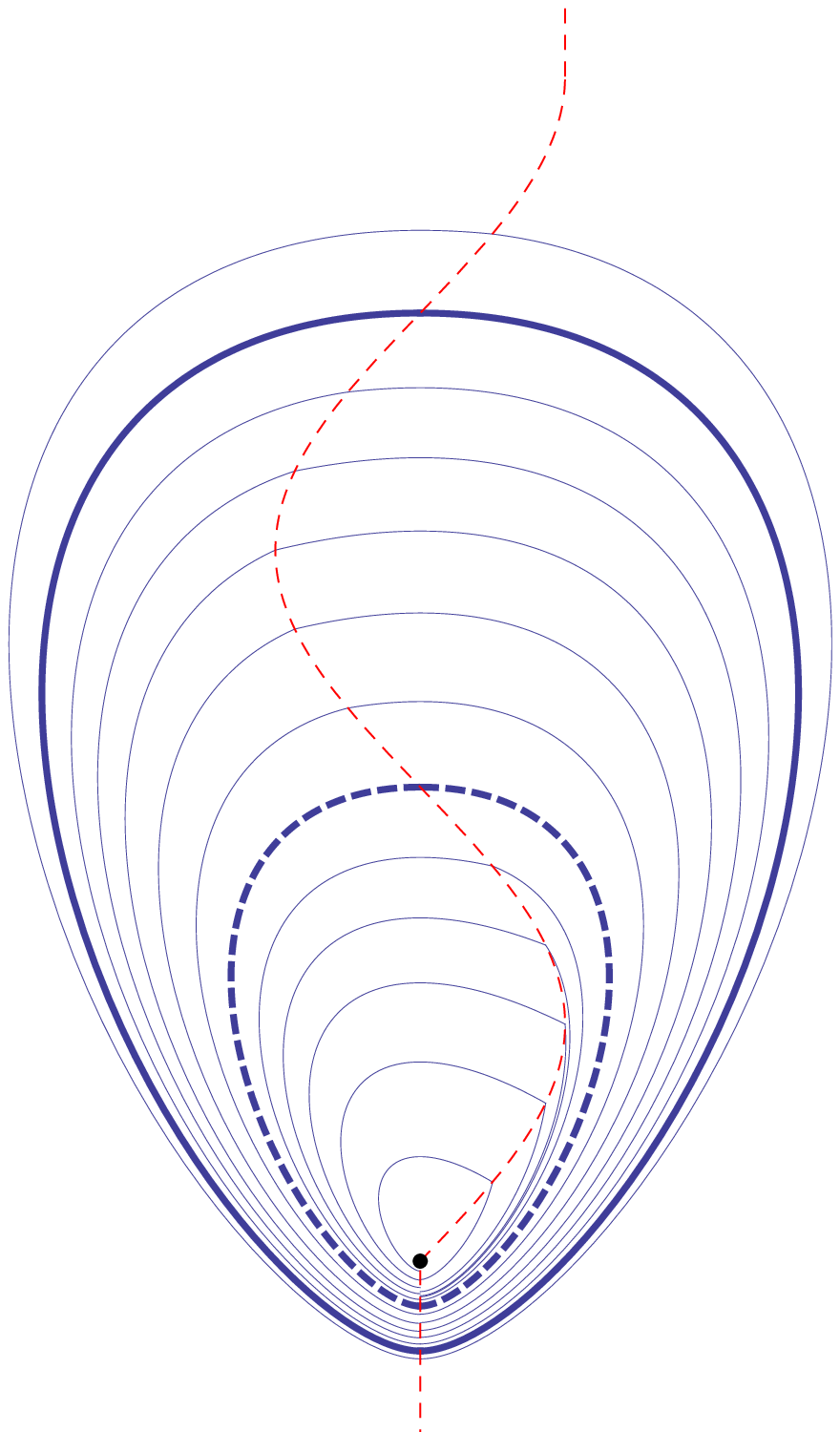}
\end{minipage}
\caption{Left: The unperturbed piecewise linear center for $\gamma=0.75$. Right: Here we consider the system of Corollary \ref{co1} for $n=2$ and $\gamma=0.75$. The continuous bold (dashed) closed curve surrounding the origin represent one stable (unstable) limit cycle, while the remaining orbits are not closed any longer. The discontinuity set is represented by the dashed line crossing the $y$--axis two times for $y>0$.
}
\end{figure}

It should be noticed that when $h(y)\equiv 0$ the boundary separating the two linear zones is a straight line, namely the $y$-axis. In such a case, we have indeed a continuous vector field with a global nonlinear center at the origin, since from each side the origin is a focus and the expansion in the right part is perfectly balanced with the contraction in the left part. Furthermore, all the periodic orbits of the center are homothetic. See the left panel of Figure \ref{fig1} and Proposition 4.2 of \cite{FPT} for more details.

In the case of a non-vanishing function $h$, the discontinuity set for system \eqref{s1} is given by 
$$\Sigma=\{(h(y),y):\,y>0, h(y)\ne 0\},$$
so that the balance between expansion and contraction is lost; as it will be seen, some periodic orbits from the original center configuration can persist, becoming isolated and so leading to limit cycles.  

\smallskip

Our main result is the following.
\begin{theorem}\label{t1} Assume $\g>0$ and consider system \eqref{s1}-\eqref{vf} and the switching manifold $H(X)=0$ as in \eqref{vfborder}, where $h$ satisfies Hypotheses $H1'$, $H2'$ and $H3'$.
For a given positive real number $y^*$ there exists a periodic solution of system \eqref{s1} passing through $(h(y^*),y^*)$ if and only if $h(y^*)=0$, in this case the periodic solution cut the $y$--axis at the points $(0,y^*)$ and $(0,-e^{-\g\pi}y^*)$. Moreover if $h'(y^*)<0$ $(\,h'(y^*)>0\,)$ this periodic solution is a  stable (unstable) limit cycle.
\end{theorem}

Theorem \ref{t1} is proved in Section \ref{sec2}.

\smallskip

The Braga--Mello Conjecture \ref{c1} is a direct consequence of Theorem \ref{t1} as we can see in the following corollaries. See also the right panel of Figure \ref{fig1}.

\begin{corollary}\label{co1}
If $0<\g<\sqrt{3/5}$ and
\begin{equation}\label{h}
h(y)=\dfrac{2\g}{(\g^2+1)\pi}\left\{\begin{array}{lr}\sin(\pi y), & 0\le y\le (2n+1)/2,\\
(-1)^{n}, & y>(2n+1)/2,\end{array}\right.
\end{equation}
then system \eqref{s1} has exactly $n$ limit cycles for any $n\in\N$. These limit cycles are nested and surround the origin, which is a stable singular point of focus type. The limit cycles cut the $y$--axis at the points $(0,k)$ and $(0,-k\exp(-\pi\g))$ for $k=1,\dots,n$, being stable (unstable) for $k$ even (odd). 
\end{corollary}

Using Theorem \ref{t1} we can find some systems exhibiting exotic configurations of limit cycles. As an example we prove the following corollary, where we find a infinite sequence of limit cycles accumulating at the origin.

\begin{corollary}\label{co2}
Given $0<\al<(-1+\sqrt{3})/2$, if $\g=1$ and $h(y)=\al y^2\sin(1/y)$ for $y>0$ with $h(0)=0$, then for each $k=1,2,\cdots$ there exists a limit cycle of system \eqref{s1} cutting the $y$--axis at the points $(0,1/(k\pi))$ and $(0,-e^{-\pi}/(k\pi))$ being stable (unstable) for $k$ even (odd). These limit cycles are nested and surround the origin, which is a stable singular point of focus type.
\end{corollary}

Corollaries \ref{co1} and \ref{co2} are proved in Section \ref{sec2}. Note that the function $h$ is bounded and that its upper bound can be taken as small as desired. Thus, we do not need a big perturbation to obtain as much limit cycles as wanted. Also note that the oscillating line used to define the discontinuity set or switching manifold has the same effect for getting several limit cycles than the one achieved in continuous Li\'enard systems with an oscillating function, see \cite{LP} and \cite{LPZ}. 

To finish, we emphasize that with a suitable choice of function $h$ one can also get as much semi-stable limit cycles as you want. We call semi-stable limit cycles the isolated periodic orbits that are stable from the interior and unstable from the exterior or vice versa. Thus, we next state our last result. 

\begin{corollary}\label{co3}
If $0<\g<\sqrt{3/13}$ and
\begin{equation}\label{h}
h(y)=\dfrac{2\g}{(\g^2+1)\pi}\left\{\begin{array}{lr}1-\cos(\pi y), & 0\le y\le 2n+1,\\
2, & y>2n+1,\end{array}\right.
\end{equation}
then system \eqref{s1} has exactly $n$ limit cycles for any $n\in\N$. These limit cycles are nested and surround the origin, which is a stable singular point of focus type. The limit cycles cut the $y$--axis at the points $(0,2k)$ and $(0,-2k\exp(-\pi\g))$ for $k=1,\dots,n$, being all of semi-stable type. 
\end{corollary}

Corollary \ref{co3} can be proved in a very similar way than Corollary \ref{co1}; in fact, its proof (to be omitted for sake of brevity) is even easier since the function $h$ is non-negative. In showing the semi-stable character of limit cycles, first two statements of Lemmma \ref{L2} should be taken into account, see below. 

\section{Proofs of Theorem \ref{t1} and Corollaries \ref{co1} and \ref{co2}}\label{sec2}
The proof of Theorem \ref{t1} is made by constructing a displacement function of points of kind $(h(y),y)$. Since system \eqref{s1} has a focus at the origin in the both sides, we obtain this displacement function by computing the difference between the position of the fist return to the section $\{x=0\,,\,y<0\}$ in forward time and the position of the fist return to the section $\{x=0\,,\,y<0\}$ in backward time considering the flow starting in $(h(y),y)$.  

\begin{proof}[Proof of Theorem \ref{t1}]
We start by computing
\begin{equation*}
\langle \nabla H(h(y),y)\,,\,G^{\pm}(h(y),y)\rangle=(1,-h'(y))^T\left(\begin{array}{l}\pm 2\g h(y)-y\\(\g^2+1)h(y)\end{array}\right),
\end{equation*}
so that, from Hypotheses $H2'$ and $H3'$, we get 
\[
\begin{array}{l}
\langle \nabla H(h(y),y)\,,\,G^+(h(y),y)\rangle=-y+h(y)\left(2\g-(1+\g^2)h'(y)\right)<0,\quad\textrm{and}\vspace{0.2cm}\\
\langle \nabla H(h(y),y)\,,\,G^-(h(y),y)\rangle=-y+h(y)\left(-2\g-(1+\g^2)h'(y)\right)<0.
\end{array}
\]

Therefore for $y>0$ the flow of system \eqref{s1} in all points $(h(y),y)$ crosses always $\Sigma$ from the right to the left, all becoming crossing points, in the usual terminology of Filippov systems, see  \cite{KRG}. In other words, excepting at the origin, the two vector fields have with respect to $\Sigma$ nontrivial normal components of the same sign.

\smallskip 

Let $\f^+(t,x,y)=\left(\f^+_1(t,x,y),\f^+_2(t,x,y)\right)$
be the solution of system \eqref{s1} for $x>0$ such that
$\f^+(0,x,y)=(x,y)$, and let
$\f^-(t,x,y)=\big(\f^-_1(t,x,y), \f^-_2(t,x,y)\big)$ be
the solution of system \eqref{s1} for $x<0$ such that
$\f^-(0,x,y)=(x,y)$. Since system \eqref{s1} is piecewise linear, this solutions can be easily computed as
\begin{equation}\label{solutions}
\begin{array}{l}
\f^{\pm}_1(t,x,y)=e^{\pm\g t} \left[(\pm\g x-y)\sin t +x \cos t\right],\vspace{0.2cm}\\

\f^{\pm}_2(t,x,y)=e^{\pm\g t} \left[\left(\g^2 x\mp\g y+x\right)\sin t +y \cos t\right].
\end{array}
\end{equation}

Let $\pi/2<t_L(y)<3\pi/2$ be the smallest positive time such that $\f^-_1(t_L(y),h(y)$ $,y)=0$ and $\f^-_2(t_L(y),h(y),y)<0$. We note that $t_L(y)>\pi$ if $h(y)>0$, $t_L(y)=\pi$ if $h(y)=0$, and $t_L(y)<\pi$ if $h(y)<0$. Similarly let $-3\pi/2<t_R(y)<-\pi/2$ be the biggest
negative time such that $\f^+_1(t_R(y),h(y),y)=0$ and $\f^+_2(t_R(y),h(y),y)$ $<0$. We note that $t_R(y)>-\pi$ if $h(y)>0$, $t_R(y)=-\pi$ if $h(y)=0$, and $t_R(y)<-\pi$ if $h(y)>0$. From hypothesis $H1'$ we have $|h(y)|<y/\g$, so $y+\g{h(y)}>0$ and $y-\g{h(y)}>0$, therefore we can easily compute $t_L(y)$ and $t_R(y)$ as
\[
\begin{array}{l}
t_L(y)=\pi+\arctan\left(\dfrac{{h(y)}}{y+\g{h(y)}}\right),\vspace{0.2cm}\\

t_R(y)=-\pi+\arctan\left(\dfrac{{h(y)}}{y-\g{h(y)}}\right).
\end{array}
\]
Accordingly, we have 
$$\cos(t_L(y))=-\dfrac{y+\g{h(y)}}{\sqrt{(y+\g{h(y)})^2+h(y)^2}},\ \  \sin(t_L(y))=-\dfrac{h(y)}{\sqrt{(y+\g{h(y)})^2+h(y)^2}},$$
and
$$\cos(t_R(y))=-\dfrac{y-\g{h(y)}}{\sqrt{(y-\g{h(y)})^2+h(y)^2}},\ \  \sin(t_R(y))=-\dfrac{h(y)}{\sqrt{(y-\g{h(y)})^2+h(y)^2}}.$$
Finally, substituting in \eqref{solutions} and after some standard manipulations, we obtain
$$\f^-_2(t_L(y),h(y),y)=-e^{-\g\pi-\g\arctan\left(\dfrac{{h(y)}}{y+\g{h(y)}}\right)}\sqrt{(y+\g{h(y)})^2+h(y)^2},$$
and
$$\f^+_2(t_R(y),h(y),y)=-e^{-\g\pi+\g\arctan\left(\dfrac{{h(y)}}{y-\g{h(y)}}\right)}\sqrt{(y-\g{h(y)})^2+h(y)^2}.$$
We construct now the displacement function as $f(y)=\f^-_2(t_L(y),h(y),y)-\f^+_2(t_R(y),h(y),y)$, thus
\[
\begin{array}{rl}
f(y)=&e^{-\g\pi+\g\arctan\left(\dfrac{{h(y)}}{y-\g{h(y)}}\right)}\sqrt{y^2-2\g y{h(y)}+\left(1+\g^2\right){h(y)^2}}\vspace{0.2cm}\\

&-e^{-\g\pi-\g\arctan\left(\dfrac{{h(y)}}{y+\g{h(y)}}\right)}\sqrt{y^2+2\g y{h(y)}+\left(1+\g^2\right){h(y)^2}}.
\end{array}
\]

If $y^*>0$ is such that $h(y^*)=0$ it is easy to see that $f(y^*)=0$. Therefore there exists a periodic solutions passing through $(h(y),y)$.  The following auxiliary results, where we prove a little bit more than needed for Theorem \ref{t1}, can be easily shown under the previous hypotheses.
\begin{lemma} \label{L1} If for $y>0$ we have $h(y)>0$ $(h(y)<0)$ then $f(y)>0$ $(f(y)<0)$. \end{lemma}
\begin{proof}
We show first that if $h(y)>0$ then $f(y)>0$. To see this, we consider for a fixed $y>0$ the function
\[
\begin{array}{rl}
F_y(x)=&e^{\g\arctan\left(\dfrac{{x}}{y-\g{x}}\right)}\sqrt{y^2+2\g y{x}+\left(1+\g^2\right){x^2}}\vspace{0.2cm}\\

&-e^{-\g\arctan\left(\dfrac{{x}}{y+\g{x}}\right)}\sqrt{y^2-2\g y{x}+\left(1+\g^2\right){x^2}}.
\end{array}
\]
Clearly $f(y)=e^{-\pi\g}F_y(h(y))$. We note that $F_y(0)=0$ and 
$$F_y(x)=\delta_y(x)-\delta_y(-x),$$ where

$$\delta_y(x)=e^{\g\arctan\left(\dfrac{{x}}{y-\g{x}}\right)}\sqrt{y^2+2\g y{x}+\left(1+\g^2\right){x^2}}.$$
Since we are dealing with a difference of two positive terms, for determining its sign we can work with the difference of squares, avoiding so to deal with square roots. Now the derivative $\delta_y(x)^2-\delta_y(-x)^2$ with respect to $x$ simplifies to
$$2x\left(1+\g^2\right)e^{2\g\arctan\left(\dfrac{{x}}{y-\g{x}}\right)}-2x\left(1+\g^2\right)e^{-2\g\arctan\left(\dfrac{{x}}{y+\g{x}}\right)},$$
which is obviously positive for all $0<x<y/\g$. Then $F_y(x)$ is monotone increasing for the same range, and we can assure that $f(y)=e^{-\pi\g}F_y(h(y))$ is positive.

Since $F_y(-x)=-F_y(x)$, the case $h(y)<0$ is a direct consequence of the above reasoning and the lemma follows.
\end{proof}
\begin{lemma} \label{L2} Assume $y^*>0$ such that $h(y^*)=0$. The following statements hold.
\begin{itemize}
\item[(i)] If there exists $\varepsilon>0$ such that $h(y)<0$ $(h(y)>0)$ for $y^*-\varepsilon<y<y^*$ then the periodic orbit passing for $(0,y^*)$ is stable (unstable) from the interior.
\item[(ii)] If there exists $\varepsilon>0$ such that $h(y)>0$ $(h(y)<0)$ for $y^*<y<y^*+\varepsilon$ then the periodic orbit passing for $(0,y^*)$ is stable (unstable) from the exterior.
\item[(iii)] If $h'(y^*)>0$ $(h'(y^*)<0)$ then there exist a periodic solution passing for $(0,y^*)$ which is a stable (unstable) limit cycle.
\end{itemize}
\end{lemma}
\begin{proof}
The three statements follow from the standard properties of displacement function and Lemma \ref{L1}. 

In the case of statement (iii), an alternative proof can be obtained by direct computations of derivatives of the displacement function $f$ at $y^*$. The expressions are quite long but simplify a lot after substituting $h(y^*)=0$. One obtains $f'(y^*)=0$, so that the limit cycles are non-hyperbolic and we need to resort to successive derivatives, getting $f''(y^*)=0$ and \[
f'''(y^*)=\dfrac{8\g(1+\g^2)e^{-\pi\g} h'(y^*)^3}{y^2}.
\]
We conclude again  that the periodic solution is a stable limit cycle if $h'(y^*)>0$, and an unstable limit cycle if $h'(y^*)<0$.
\end{proof}

\smallskip

From Lemma \ref{L2}, Theorem \ref{t1} is shown.  \end{proof}

%He movido lo que hab'a aqu' al final, despuŽs del \end{document}, Enrique

It follows the proofs of corollaries.

\begin{proof}[Proof of Corollary \ref{co1}] It is easy to see that $$|h'(y)|\le \dfrac{2\g}{\g^2+1},$$
and to fulfill Hypothesis $H1'$ we should need 
$$\dfrac{2\g}{\g^2+1}<\dfrac{1}{\g},
$$
which is true for all $0<\g<1$. 

\smallskip

Furthermore, for $y\leq (2n+1)/2$ we have
\[
h(y)(2\g\pm(1+\g^2)h'(y))=\dfrac{4\g^2}{(\g^2+1)\pi}\sin(\pi y)(1\pm\cos(\pi y)).
\]
We note that for $0<\g<\sqrt{3/5}$ the inequality $8\g^2/(\pi+\pi\g^2)<1$ hold. Again using that $|\sin(y)|<y$ for $y>0$ we obtain
\[
\begin{array}{l}
h(y)(2\g-(1+\g^2)h'(y))<\dfrac{8\g^2}{(1+\g^2)\pi}y<y,\quad\textrm{and}\vspace{0.2cm}\\
h(y)(2\g+(1+\g^2)h'(y))>-\dfrac{8\g^2}{(1+\g^2)\pi}y>-y.
\end{array}
\]
So the Hypotheses $H2'$ and $H3'$ hold for $y\leq (2n+1)/2$.

\smallskip

Now for $y>(2n+1)/2\ge 3/2$ we have $h'(y)=0$, so that
\[
h(y)(2\g\pm(1+\g^2)h'(y))=\dfrac{4\g^2}{(\g^2+1)\pi}(-1)^{n},\]
and the inequalities in Hypotheses $H2'$ and $H3'$ trivially hold.

\smallskip

Computing the zeros of the function $h$ and applying Theorem \ref{t1} we conclude that system \eqref{s1} has exactly $n$ limit cycles for any given $n\in\N$ cuting the $y$-axis at the points $(0,k)$ and $(0,-k\exp(-\pi\g))$ for $k=1,\dots,n$.
\end{proof}
 
\begin{proof}[Proof of Corollary \ref{co2}]
We note first that the hypothesis $0<\al<(-1+\sqrt{3})/2$ implies that $2\al(1+\al)<1$. Thus, we also have $\al<1$, so that using the inequality $|\sin(1/y)|<1/y$ for $y>0$, it is easy to see that $|h(y)|<\al y<y$, and Hypothesis $H1'$ holds.

\smallskip
Since $\gamma=1$, we have
\[
h(y)(2\g\pm(1+\g^2)h'(y))=2h(y)(1\pm h'(y)),
\]
and then
\[
2h(y)(1\pm h'(y))=
2\al y^2\sin\left(\dfrac{1}{y}\right)\mp2\al^2y^2\sin\left(\dfrac{1}{y}\right)\cos\left(\dfrac{1}{y}\right)\pm 4\al^2 y^3\sin^2\left(\dfrac{1}{y}\right).
\]
So using again the inequality $|\sin(1/y)|<1/y$ for $y>0$ we obtain
\[
\begin{array}{l}
2h(y)(1-h'(y))<2\al y+2\al^2y=2\al(1+\al)y<y,\quad\textrm{and}\vspace{0.2cm}\\
2h(y)(1+h'(y))>-2\al y-2\al^2y=-2\al(1+\al)y>-y.
\end{array}
\]
Hence the inequalities in Hypotheses $H2'$ and $H3'$ hold for $y>0$.

\smallskip

Computing the zeros of the function $h$ and applying Theorem \ref{t1}, we conclude the proof of the corollary.
\end{proof}

\section*{Acknowledgements}

The first author is partially supported by a FAPESP grant
2013/16492--0 and by a CAPES CSF-PVE grant 88881.030454/2013-01. The second author was supported by MINECO/FEDER grant MTM2012-31821 and by the {\it Consejer\'{\i}a de Econom\'{\i}a, Innovaci\'on, Ciencia y Empleo de la Junta de Andaluc\'{\i}a} under grant P12-FQM-1658.


\begin{thebibliography}{99}

\bibitem[Braga \& Mello, 2014]{BM} Braga, D.C. \& Mello, L.F. [2014] ``More than Three Limit Cycles in Discontinuous Piecewise Linear Differential Systems with Two Zones in the Plane,'' {\it International J. of Bifurcation and Chaos} {\bf 24} No. 4, 1450056 (10 pages).

\bibitem[Freire {\it et al}, 2012]{FPT} Freire, E., Ponce E. \& Torres, F. [2012] ``Canonical Discontinuous Planar Piecewise Linear Systems,'' {\it SIAM J. Appl. Dyn. Syst.} {\bf 11}, 181--211.

\bibitem[Freire {\it et al}, 2014]{FPT2014} Freire, E., Ponce E. \& Torres, F. [2014] ``The Discontinuous Matching of Two Planar Linear Foci Can Have Three Nested Crossing Limit Cycles,'' {\it Publ. Mat.} Vol. extra, 221--253.

\bibitem[Han \& Zhang, 2010]{HanZhang} Han, M. \& Zhang, W. [2010] ``On Hopf Bifurcation in Non-smooth Planar Systems,''
{\it J. Differential Equations} \textbf{248}, 2399--2416.

\bibitem[Huan \& Yang, 2012]{HuanYang} S. Huan, X. Yang [2012]
``On the Number of Limit Cycles in General Planar Piecewise Linear Systems,''
{\it Discrete and Continuous Dynamical Systems-A} \textbf{32}, 2147--2164.

\bibitem[Kuznetsov {\it et al}, 2003]{KRG} Kuznetsov, Y. A.,  Rinaldi, S. \& Gragnani, A. [2003] ``One-parameter Bifurcations in Planar Filippov Systems'' {\it International J. of Bifurcation and Chaos} {\bf 13} No. 8, 215--2188.

\bibitem[Llibre \& Ponce, 2003]{LP} Llibre, J. \& Ponce, E. [2003] ``Piecewise Linear Feedback Systems with Arbitrary Number of Limit Cycles'' {\it Int. J. Bifurcation and Chaos} {\bf 12}, 895--904.

\bibitem[Llibre \& Ponce, 2012]{LlibrePonce2012} Llibre, J. \& Ponce, E. [2012]
``Three Nested Limit Cycles in Discontinuous Piecewise Linear Differential Systems,''
{\it Dynamics of Continuous Discrete and Impulsive Systems Series B} \textbf{19}, 325--335.

\bibitem[Llibre {\it et al}, 2012]{LPZ} Llibre, J., Ponce, E. \& Zhang, X. [2012] ``Existence of Piecewise Linear Differential Systems with Exactly $n$ Limit Cycles for all $n\in\N$'' {\it Nonlin. Anal.} {\bf 54}, 324--335.
\end{thebibliography}
\end{document}